\documentclass[runningheads]{llncs}
\usepackage{graphicx}
\usepackage{makeidx}
\usepackage{amsmath}
\usepackage{amssymb}

\DeclareMathOperator{\ord}{ord} 
 
\DeclareMathOperator{\trdeg}{trdeg}\DeclareMathOperator{\type}{type}

\begin{document}
\def\D{\displaystyle}

\title{Hilbert-type dimension polynomials of intermediate
difference-differential field extensions\thanks{Supported by the NSF grant CCF-1714425}}
\titlerunning{Hilbert-type dimension polynomials}
%
\author{Alexander Levin}
\authorrunning{A. Levin}
%
\institute{The Catholic University of America, Washington, DC 20064, USA\\
\email{levin@cua.edu}\\
\url{https://sites.google.com/a/cua.edu/levin}}
\maketitle              

\begin{abstract}
Let $K$ be an inversive difference-differential field and $L$ a (not
necessarily inversive) finitely generated difference-differential
field extension of $K$. We consider the natural filtration of the
extension $L/K$ associated with a finite system $\eta$ of its
difference-differential generators and prove that for any
intermediate difference-differential field $F$, the transcendence
degrees of the components of the induced filtration of $F$ are
expressed by a certain numerical polynomial $\chi_{K,
F,\eta}(t)$. This polynomial is closely connected with the dimension
Hilbert-type polynomial of a submodule of the module of K\"ahler
differentials $\Omega_{L^{\ast}|K}$ where $L^{\ast}$ is the
inversive closure of $L$. We prove some properties of polynomials
$\chi_{K, F,\eta}(t)$ and use them for the study of the Krull-type
dimension of the extension $L/K$. In the last part of the paper, we
present a generalization of the obtained results to multidimensional
filtrations of $L/K$ associated with partitions of the sets of basic
derivations and translations.

\medskip

\end{abstract}

\section{Introduction}

Dimension polynomials associated with finitely generated
differential field extensions were introduced by E. Kolchin in
\cite{K1}; their properties and various applications can be found in
his fundamental monograph \cite[Chapter 2]{K2}. A similar technique
for difference and inversive difference field extensions was
developed in \cite{Levin1}, \cite{Levin2}, \cite{Levin5},
\cite{Levin6} and some other works of the author. Almost all known
results on differential and difference dimension polynomials can be
found in \cite{KLMP} and \cite{Levin3}. One can say that the role of
dimension polynomials in differential and difference algebra is
similar to the role of Hilbert polynomials in commutative algebra
and algebraic geometry. The same can be said about dimension
polynomials associated with difference-differential algebraic
structures. They appear as generalizations of their differential and
difference counterparts and play a key role in the study of
dimension of difference-differential modules and extensions of
difference-differential fields. Existence theorems, properties and
methods of computation of univariate and multivariate
difference-differential dimension polynomials can be found in
\cite{LM}, \cite[Chapters 6 and 7]{KLMP}, \cite{Levin7}, \cite{ZW1}
and \cite{ZW2}.

In this paper we prove the existence and obtain some properties of a
univariate dimension polynomial associated with an intermediate
difference-differential field of a finitely generated
difference-differential field extension (see Theorem 2 that can be
considered as the main result of the paper). Then we use the
obtained results for the study of the Krull-type dimension of such
an extension. In particular, we establish relationships between
invariants of dimension polynomials and characteristics of
difference-differential field extensions that can be expressed in
terms of chains of intermediate fields. In the last part of the
paper we generalize our results on univariate dimension polynomials
and obtain multivariate dimension polynomials associated with
multidimensional filtrations induced on intermediate
difference-differential fields. (Such filtrations naturally arise
when one considers partitions of the sets of basic derivations and
translations.) Note that we consider arbitrary (not necessarily
inversive) difference-differential extensions of an inversive
difference-differential field. In the particular case of purely
differential extensions and in the case of  inversive difference
field extensions, the existence and properties of dimension
polynomials were obtained in \cite{Levin4} and \cite{Levin6}. The
main problem one runs into while working with a non-inversive
difference (or difference-differential) field extension is that the
translations are not invertible and there is no natural difference
(respectively, difference-differential) structure on the associated
module of K\"ahler differentials. We overcome this obstacle by
considering such a structure on the module of K\"ahler differentials
associated with the inversive closure of the extension. Finally, the
results of this paper allow one to assign a dimension polynomial to
a system of algebraic difference-differential equations of the form
$f_{i} = 0$, $i\in I$ ($f_{i}$ lie in the algebra of
difference-differential polynomials $K\{y_{1},\dots, y_{n}\}$ over a
ground field $K$) such that the difference-differential ideal $P$
generated by the left-hand sides is prime and the solutions of the
system should be invariant with respect to the action of a group $G$
that commutes with basic derivations and translations. As in the
case of systems of differential or difference equations, the
dimension polynomial of such a system is defined as the dimension
polynomial of the subfield of the difference-differential quotient
field $K\{y_{1},\dots, y_{n}\}/P$ whose elements remain fixed under
the action of $G$. Using the correspondence between dimension
polynomials and Einstein's strength of a system of algebraic
differential or difference equations established in \cite{MP} and
\cite[Chapter 6]{KLMP} (this characteristic of a system of PDEs
governing a physical field was introduced in \cite{E}), one can
consider this dimension polynomial as an expression of the
Einstein's strength of a system of difference-differential equations
with group action.

\section{Preliminaries}

Throughout the paper $\mathbb{Z}$, $\mathbb{N}$ and $\mathbb{Q}$
denote the sets of all integers, all non-negative integers and all
rational numbers, respectively. As usual, $\mathbb{Q}[t]$ will
denote the ring of polynomials in one variable $t$ with rational
coefficients. By a ring we always mean an associative ring with a
unity. Every ring homomorphism is unitary (maps unit onto unit),
every subring of a ring contains the unity of the ring. Every module
is unitary and every algebra over a commutative ring is unitary as
well. Every field is supposed to have characteristic zero.

A {\em difference-differential ring} is a commutative ring $R$
considered together with finite sets $\Delta = \{\delta_{1},\dots,
\delta_{m}\}$ and $\sigma = \{\alpha_{1},\dots, \alpha_{n}\}$ of
derivations and injective endomorphisms of $R$, respectively, such
that any two mappings of the set $\Delta\bigcup\sigma$ commute. The
elements of the set $\sigma$ are called {\em translations} and the
set $\Delta\bigcup\sigma$ will be referred to as a {\em basic set}
of the difference-differential ring $R$, which is also called a
$\Delta$-$\sigma$-ring. We will often use prefix $\Delta$-$\sigma$-
instead of the adjective ''difference-differential''. If all
elements of $\sigma$ are automorphisms of $R$, we say that the
$\Delta$-$\sigma$-ring $R$ is {\em inversive}. In this case we set
$\sigma^{\ast} = \{\alpha_{1},\dots, \alpha_{n},
\alpha^{-1}_{1},\dots, \alpha^{-1}_{n}\}$ and call $R$ a
$\Delta$-$\sigma^{\ast}$-ring.

If a $\Delta$-$\sigma$-ring $R$ is a field, it is called a {\em
difference-differential field} or a $\Delta$-$\sigma$-field. If $R$
is inversive, we say that $R$ is a $\Delta$-$\sigma^{\ast}$-field.

In what follows, $\Lambda$ will denote the free commutative
semigroup of all power products $\lambda = \delta_{1}^{k_{1}}\dots
\delta_{m}^{k_{m}}\alpha_{1}^{l_{1}}\dots \alpha_{n}^{l_{n}}$ where
$k_{i}, l_{j}\in\mathbb{N}$ ($1\leq i\leq m,\, 1\leq j\leq n$).
Furthermore,  $\Theta$ and $T$ will denote the commutative
semigroups of power products $\delta_{1}^{k_{1}}\dots
\delta_{m}^{k_{m}}$ and $\alpha_{1}^{l_{1}}\dots \alpha_{n}^{l_{n}}$
($k_{i}, l_{j}\in\mathbb{N}$), respectively. If $\lambda =
\delta_{1}^{k_{1}}\dots \delta_{m}^{k_{m}}\alpha_{1}^{l_{1}}\dots
\alpha_{n}^{l_{n}}\in\Lambda$, we define the order of $\lambda$ as
$\ord\,\lambda = \sum_{i=1}^{m}k_{i}+\sum_{j=1}^{n}l_{j}$ and set
$\Lambda(r) = \{\lambda\in\Lambda\,|\,\ord\,\lambda\leq r\}$ for any
$r\in\mathbb{N}$.

If the elements of $\sigma$ are automorphisms, then $\Lambda^{\ast}$
and $\Gamma$ will denote the free commutative semigroup of all power
products $\mu = \delta_{1}^{k_{1}}\dots
\delta_{m}^{k_{m}}\alpha_{1}^{l_{1}}\dots \alpha_{n}^{l_{n}}$ with
$k_{i}\in\mathbb{N}$, $l_{j}\in\mathbb{Z}$ and the free commutative
group of power products $\gamma=\alpha_{1}^{l_{1}}\dots
\alpha_{n}^{l_{n}}$ with $l_{1},\dots, l_{n}\in\mathbb{Z}$,
respectively. The order of such elements $\mu$ and $\gamma$ are
defined as $\ord\,\lambda =
\sum_{i=1}^{m}k_{i}+\sum_{j=1}^{n}|l_{j}|$ and $\ord\,\gamma =
\sum_{j=1}^{n}|l_{j}|$, respectively. We also set $\Lambda^{\ast}(r)
= \{\mu\in\Lambda^{\ast}\,|\,\ord\,\mu\leq r\}$ ($r\in\mathbb{N}$).

A subring (ideal) $S$ of a $\Delta$-$\sigma$-ring $R$ is said to be
a difference-differential  (or $\Delta$-$\sigma$-) subring of $R$
(respectively, difference-differential (or $\Delta$-$\sigma$-) ideal
of $R$) if $S$ is closed with respect to the action of any operator
of $\Delta\bigcup\sigma$. In this case the restriction of a mapping
from $\Delta\bigcup\sigma$ on $S$ is denoted by same symbol. If $S$
is a $\Delta$-$\sigma$-subring $R$, we also say that $R$ is a
$\Delta$-$\sigma$-overring of $S$. If $S$ is a
$\Delta$-$\sigma$-ideal of $R$ and for any $\tau\in T$, the
inclusion $\tau(a)\in S$ implies that $a\in S$, we say that the
$\Delta$-$\sigma$-ideal $S$ is {\em reflexive} or that $S$ is a
$\Delta$-$\sigma^{\ast}$-ideal of $R$.

If $L$ is a $\Delta$-$\sigma$-field and $K$ a subfield of $L$ which
is also a $\Delta$-$\sigma$-subring of $L$, then  $K$ is said to be
a $\Delta$-$\sigma$-subfield of $L$; $L$, in turn, is called a
difference-differential (or $\Delta$-$\sigma$-) field extension or a
$\Delta$-$\sigma$-overfield of $K$. In this case we also say that we
have a $\Delta$-$\sigma$-field extension $L/K$.

If $R$ is a $\Delta$-$\sigma$-ring and $S\subseteq R$, then the
intersection of all $\Delta$-$\sigma$-ideals of $R$ containing the
set $S$ is, obviously, the smallest $\Delta$-$\sigma$-ideal of $R$
containing $S$. This ideal is denoted by $[S]$; as an ideal, it is
generated by the set $\{\lambda(x)\,|\,x\in S, \lambda\in
\Lambda\}$. If $S$ is finite, $S = \{x_{1},\dots, x_{k}\}$, we say
that the $\Delta$-$\sigma$-ideal $I = [S]$ is finitely generated,
write $I = [x_{1},\dots, x_{k}]$ and call $x_{1},\dots, x_{k}$
$\Delta$-$\sigma$-generators of $I$.

If $K$ is a $\Delta$-$\sigma$-subfield of the
$\Delta$-$\sigma$-field $L$ and $S\subseteq L$, then the
intersection of all $\Delta$-$\sigma$-subfields of $L$ containing
$K$ and $S$ is the unique $\Delta$-$\sigma$-subfield of $L$
containing $K$ and $S$ and contained in every
$\Delta$-$\sigma$-subfield of $L$ with this property. It is denoted
by $K\langle S\rangle$. If $S$ is finite, $S =
\{\eta_{1},\dots,\eta_{s}\}$ we write
$K\langle\eta_{1},\dots,\eta_{s}\rangle$ for $K\langle S\rangle$ and
say that this is a finitely generated $\Delta$-$\sigma$-extension of
$K$ with the set of $\Delta$-$\sigma$-generators
$\{\eta_{1},\dots,\eta_{s}\}$.  It is easy to see that $K\langle
\eta_{1},\dots,\eta_{s}\rangle$ coincides with the field
$K(\{\lambda \eta_{i}\,|\,\lambda \in \Lambda, 1\leq i\leq s\})$.
(If there might be no confusion, we often write $\lambda\eta$ for
$\lambda(\eta)$ where $\lambda\in\Lambda$ and $\eta$ is an element
of a $\Delta$-$\sigma$-ring.)

Let $R_{1}$ and $R_{2}$ be two difference-differential rings with
the same basic set $\Delta\bigcup\sigma$. (More rigorously, we
assume that there exist injective mappings of the sets $\Delta$ and
$\sigma$ into the sets of derivations and automorphisms of the rings
$R_{1}$ and $R_{2}$, respectively, such that the images of any two
elements of $\Delta\bigcup\sigma$ commute. We will
denote the images of elements of $\Delta\bigcup\sigma$ under these
mappings by the same symbols $\delta_{1},\dots, \delta_{m},
\alpha_{1},\dots, \alpha_{n}$). A ring homomorphism $\phi: R_{1}
\longrightarrow R_{2}$ is called a {\em difference-differential} (or
$\Delta$-$\sigma$-) {\em homomorphism} if $\phi(\tau a) = \tau
\phi(a)$ for any $\tau \in \Delta\bigcup\sigma$, $a\in R$. It is
easy to see that the kernel of such a mapping is a
$\Delta$-$\sigma^{\ast}$-ideal of $R_{1}$.

If $R$ is a $\Delta$-$\sigma$-subring of a $\Delta$-$\sigma$-ring
$R^{\ast}$ such that the elements of $\sigma$ act as automorphisms
of $R^{\ast}$ and for every $a\in R^{\ast}$ there exists $\tau\in T$
such that $\tau(a)\in R$, then the $\Delta$-$\sigma^{\ast}$-ring
$R^{\ast}$ is called the {\em inversive closure} of $R$.

The proof of the following result can be obtained by mimicking the
proof of the corresponding statement about inversive closures of
difference rings, see \cite[Proposition 2.1.7]{Levin3}.
\begin{proposition}
\noindent{\em (i)}\, Every $\Delta$-$\sigma$-ring has an inversive closure.

\noindent{\em (ii)}\, If $R^{\ast}_{1}$ and $R^{\ast}_{2}$ are two
inversive closures of a $\Delta$-$\sigma$-ring $R$, then there
exists a $\Delta$-$\sigma$-isomorphism of $R^{\ast}_{1}$ onto
$R^{\ast}_{2}$ that leaves elements of $R$ fixed. .

\noindent{\em (iii) }\, If a $\Delta$-$\sigma$-ring $R$ is a
$\Delta$-$\sigma$-subring of a $\Delta$-$\sigma^{\ast}$-ring $U$,
then $U$ contains an inversive closure of $R$.

\noindent{\em (iv)}\,  If a $\Delta$-$\sigma$-ring $R$ is
a field, then its inversive closure is also a field.
\end{proposition}

If $K$ is an inversive difference-differential field and $L =
K\langle \eta_{1},\dots,\eta_{s}\rangle$, then the inversive closure
of $L$ is denoted by $K\langle
\eta_{1},\dots,\eta_{s}\rangle^{\ast}$. Clearly, this
$\Delta$-$\sigma^{\ast}$-field coincides with the field $K(\{\mu
\eta_{i} | \mu\in\Lambda^{\ast}, 1\leq i\leq s\})$.

Let $R$ be a $\Delta$-$\sigma$-ring and $U = \{u_{i}\,|\,i\in I\}$ a
family of elements of some $\Delta$-$\sigma$-overring of $R$. We say
that the family $U$ is $\Delta$-$\sigma$-{\em algebraically
dependent} over $R$, if the family $\left\{\lambda
u_{i}\,|\,\lambda\in\Lambda,\, i\in I\right\}$ is algebraically
dependent over $R$. Otherwise, the family $U$ is said to be
$\Delta$-$\sigma$-{\em algebraically independent} over $R$.

If $K$ is a $\Delta$-$\sigma$-field and $L$ a
$\Delta$-$\sigma$-field extension of $K$, then a set $B\subseteq L$
is said to be a $\Delta$-$\sigma$-{\em transcendence basis} of $L$
over $K$ if $B$ is $\Delta$-$\sigma$-algebraically independent over
$K$ and every element $a\in L$ is $\Delta$-$\sigma$-algebraic over
$K\langle B\rangle$ (that is, the set $\{\lambda
a\,|\,\lambda\in\Lambda\}$ is algebraically dependent over $K\langle
B\rangle$). If $L$ is a finitely generated $\Delta$-$\sigma$-field
extension of $K$, then all $\Delta$-$\sigma$-transcendence bases of
$L$ over $K$ are finite and have the same number of elements (the
proof of this fact can be obtained by mimicking the proof of the
corresponding properties of difference transcendence bases, see
\cite[Section 4.1]{Levin3}). In this case, the number of elements of
any $\Delta$-$\sigma$-transcendence basis is called the {\em
difference-differential}  (or $\Delta$-$\sigma$-) {\em transcendence
degree} of $L$ over $K$ (or the $\Delta$-$\sigma$-transcendence
degree of the extension $L/K$); it is denoted by
$\Delta$-$\sigma$-$\trdeg_{K}L$.

The following theorem proved in \cite{LM} generalizes the Kolchin's
theorem on differential dimension polynomial (see \cite[Chapter II,
Theorem 6]{K2}) and also the author's theorems on dimension
polynomials of difference and inversive difference field extensions
(see \cite[Theorems 4.2.1 and 4.2.5]{Levin3}).

\begin{theorem} With the above notation, let $L = K\langle \eta_{1},\dots,\eta_{s}\rangle$
be a $\Delta$-$\sigma$-field extension of a $\Delta$-$\sigma$-field
$K$ generated by a finite set $\eta =\{\eta_{1},\dots,\eta_{s}\}$.
Then there exists a polynomial $\chi_{\eta|K}(t)\in\mathbb{Q}[t]$
such that

\noindent{\em (i)}\, $\chi_{\eta|K}(r) =
\trdeg_{K}K(\{\lambda\eta_{j}\,|\,\lambda\in \Lambda(r), 1\leq j\leq
s\})$ for all sufficiently large $r\in\mathbb{Z}$ (that is, there
exists $r_{0}\in\mathbb{Z}$ such that the equality holds for all
$r>r_{0}$).

\noindent{\em (ii)}\, $\deg \chi_{\eta|K} \leq m+n$ and
$\chi_{\eta|K}(t)$ can be written as \, $\chi_{\eta|K}(t) =
\D\sum_{i=0}^{m+n}a_{i}{t+i\choose i}$, where $a_{i}\in\mathbb{Z}$.

\noindent{\em (iii)}\, $d =  \deg \chi_{\eta|K}$,\, $a_{m+n}$ and $a_{d}$ do
not depend on the set of $\Delta$-$\sigma$-generators $\eta$ of
$L/K$ ($a_{m+n}=0$ if $d < m+n$). Moreover,
$a_{m+n} = \Delta$-$\sigma$-$\trdeg_{K}L$.
\end{theorem}

\noindent The polynomial $\chi_{\eta|K}(t)$ is called the {\em
$\Delta$-$\sigma$-dimension polynomial} of the
$\Delta$-$\sigma$-field extension $L/K$ associated with the system
of $\Delta$-$\sigma$-generators $\eta$ . We see that
$\chi_{\eta|K}(t)$ is a polynomial with rational coefficients that
takes integer values for all sufficiently large values of the
argument. Such polynomials are called {\em numerical}; their
properties are thoroughly described in \cite[Chapter 2]{KLMP}. The
invariants $d =  \deg \chi_{\eta|K}$ and $a_{d}$ (if $d < m+n$) are
called {\em $\Delta$-$\sigma$-type} and {\em typical
$\Delta$-$\sigma$-transcendence degree} of $L/K$; they are denoted
by $\Delta$-$\sigma$-$\type_{K}L$ and
$\Delta$-$\sigma$-$t.\trdeg_{K}L$, respectively.

\section{Dimension polynomials of intermediate difference-differential fields. The main theorem}

The following result is an essential generalization of Theorem 1.
This generalization allows one to assign certain numerical
polynomial to an intermediate $\Delta$-$\sigma$-field of a
$\Delta$-$\sigma$-field extension $L/K$ where $K$ is an inversive
$\Delta$-$\sigma$-field. (We use the notation introduced in the
previous section.)

\begin{theorem} Let $K$ be an inversive $\Delta$-$\sigma$-field with
basic set $\Delta\bigcup\sigma$ where $\Delta = \{\delta_{1},\dots,
\delta_{m}\}$ and $\sigma = \{\alpha_{1},\dots, \alpha_{n}\}$ are
the set of derivations and automorphisms of $K$, respectively. Let
$L = K\langle \eta_{1},\dots,\eta_{s}\rangle$ be a
$\Delta$-$\sigma$-field extension of $K$ generated by a finite set
$\eta =\{\eta_{1},\dots,\eta_{s}\}$. Let $F$ be an intermediate
$\Delta$-$\sigma$-field of the extension $L/K$ and for any
$r\in\mathbb{N}$, let  $F_{r} = F\bigcap
K(\{\lambda\eta_{j}\,|\,\lambda\in\Lambda(r), 1\leq j\leq s\})$.
Then there exists a numerical polynomial $\chi_{K, F,
\eta}(t)\in\mathbb{Q}[t]$ such that

\noindent{\em (i)}\, $\chi_{K,F,\eta}(r) = \trdeg_{K}F_{r}$ for all
sufficiently large $r\in\mathbb{N}$;

\noindent{\em (ii)}\, $\deg \chi_{K,F,\eta} \leq m+n$ and
$\chi_{K,F,\eta}(t)$ can be written as $\chi_{K,F,\eta}(t) =
\D\sum_{i=0}^{m+n}c_{i}{t+i\choose i}$\,\, where
$c_{i}\in\mathbb{Z}$ ($1\leq i\leq m+n$).

\noindent{\em (iii)}\, $d =  \deg \chi_{K,F,\eta}(t)$, \,$c_{m+n}$ and\,
$c_{d}$ do not depend on the set of $\Delta$-$\sigma$-generators
$\eta$ of the extension $L/K$. Furthermore, $c_{m+n} =
\Delta$-$\sigma$-$\trdeg_{K}F$.
\end{theorem}

The polynomial $\chi_{K,F,\eta}(t)$ is called a {\em
$\Delta$-$\sigma$-dimension polynomial of the intermediate field $F$
associated with the set of $\Delta$-$\sigma$-generators $\eta$ of
$L/K$}.

The proof of Theorem 2 is based on properties of difference-differential
modules and the difference-differential structure on the module of K\"ahler differentials
considered below. Similar properties in differential and difference cases
can be found in \cite{Johnson1} and \cite[Section 4.2]{Levin3}, respectively.

Let $K$ be a $\Delta$-$\sigma$-field and $\Lambda$ the semigroup of
power products of basic operators introduced in section 2. Let
$\mathcal{D}$ denote the set of all finite sums of the form
$\sum_{\lambda\in\Lambda}a_{\lambda}\lambda$ where $a_{\lambda}\in
K$ (such a sum is called a $\Delta$-$\sigma$-{\em operator} over
$K$; two $\Delta$-$\sigma$-operators are equal if and only if their
corresponding coefficients are equal). The set ${\cal{D}}$ can be
treated as a ring with respect to its natural structure of a left
$K$-module and the relationships $\delta a = a\delta + \delta(a)$,
$\alpha a = \alpha(a)\alpha$ for any $a\in K$, $\delta\in\Delta$,
$\alpha\in\sigma$ extended by distributivity. The ring ${\cal{D}}$
is said to be the {\em ring of $\Delta$-$\sigma$-operators} over
$K$.

If $A = \sum_{\lambda\in\Lambda}a_{\lambda}\lambda\in{\cal{D}}$,
then the number $\ord A = \max\{\ord \lambda\,|\,a_{\lambda}\neq
0\}$ is called the {\em order} of the $\Delta$-$\sigma$-operator
$A$. In what follows, we treat ${\cal{D}}$ as a filtered ring with
the ascending filtration $({\cal{D}}_{r})_{r\in\mathbb{Z}}$ where
${\cal{D}}_{r} = 0$ if $r < 0$ and ${\cal{D}}_{r} = \{A\in
{\cal{D}}\,|\,\ord A\leq r\}$ if $r\geq 0$.

Similarly, if a $\Delta$-$\sigma$-field $K$ is inversive and
$\Lambda^{\ast}$ is the semigroup defined in section 2, then
$\mathcal{E}$ will denote the set of all finite sums
$\sum_{\mu\in\Lambda^{\ast}}a_{\mu}\mu$ where $a_{\mu}\in K$. Such a
sum is called a $\Delta$-$\sigma^{\ast}$-operator over $K$; two
$\Delta$-$\sigma^{\ast}$-operators are equal if and only if their
corresponding coefficients are equal. Clearly, the ring
$\mathcal{D}$ of $\Delta$-$\sigma$-operators over $K$ is a subset of
${\cal{E}}$. Moreover, ${\cal{E}}$ can be treated as an overring of
$\mathcal{D}$ such that $\alpha^{-1}a = \alpha^{-1}(a)\alpha^{-1}$
for every $\alpha\in\sigma$, $a\in K$. This ring is called the {\em
ring of $\Delta$-$\sigma^{\ast}$-operators} over $K$.

The order of a $\Delta$-$\sigma^{\ast}$-operator $B =
\sum_{\mu\in\Lambda^{\ast}}a_{\mu}\mu$ is defined in the same way as
the order of a $\Delta$-$\sigma$-operator: $\ord B = \max\{\ord
\mu\,|\,a_{\mu}\neq 0\}$. In what follows the ring ${\cal{E}}$ is
treated as a filtered ring with the ascending filtration
$({\cal{E}}_{r})_{r\in\mathbb{Z}}$ such that ${\cal{E}}_{r} = 0$ if
$r < 0$ and ${\cal{E}}_{r} = \{B\in {\cal{E}}\,|\,\ord B\leq r\}$ if
$r\geq 0$.

If $K$ is a $\Delta$-$\sigma$-field, then a {\em
difference-differential module} over $K$ (also called a
$\Delta$-$\sigma$-$K$-module) is a left ${\cal{D}}$-module $M$, that
is, a vector $K$-space where elements of $\Delta\bigcup\sigma$ act
as additive mutually commuting operators such that $\delta(ax) =
a(\delta x) + \delta(a)x$ and $\alpha(ax) = \alpha(a)\alpha x$ for
any $\delta\in\Delta$, $\alpha\in \sigma$, $x\in M$, $a\in K$. We
say that $M$ is a finitely generated $\Delta$-$\sigma$-$K$-module if
$M$ is finitely generated as a left ${\cal{D}}$-module.

Similarly, if $K$ is a $\Delta$-$\sigma^{\ast}$-field, then an {\em
inversive difference-differential module} over $K$ (also called a
$\Delta$-$\sigma^{\ast}$-$K$-module) is a left ${\cal{E}}$-module
(that is, a $\Delta$-$\sigma$-$K$-module $M$ with the action of
elements of $\sigma^{\ast}$ such that $\alpha^{-1}(ax) =
\alpha^{-1}(a)\alpha^{-1}x$ for every $\alpha\in\sigma$). A
$\Delta$-$\sigma^{\ast}$-$K$-module $M$ is said to be finitely
generated if it is generated as a left ${\cal{E}}$-module by a
finite set whose elements are called
$\Delta$-$\sigma^{\ast}$-generators of $M$.

If $M$ is a $\Delta$-$\sigma$-$K$-module (respectively, a
$\Delta$-$\sigma^{\ast}$-module, if $K$ is a
$\Delta$-$\sigma^{\ast}$-field), then by a filtration of $M$ we mean
an exhaustive and separated filtration of $M$ as a ${\cal{D}}$-
(respectively, ${\cal{E}}$-) module, that is, an ascending chain
$(M_{r})_{r\in\mathbb{Z}}$ of vector $K$-subspaces of $M$ such that
${\cal{D}}_{r}M_{s}\subseteq M_{r+s}$ (respectively,
${\cal{E}}_{r}M_{s}\subseteq M_{r+s}$) for all $r, s\in \mathbb{Z}$,
$M_{r} = 0$ for all sufficiently small $r\in \mathbb{Z}$, and
$\bigcup_{r\in \mathbb{Z}}M_{r} = M$. A filtration
$(M_{r})_{r\in\mathbb{Z}}$ of a $\Delta$-$\sigma$-$K$-
(respectively, $\Delta$-$\sigma^{\ast}$-$K$) module $M$ is said to
be {\em excellent} if every $M_{r}$ is a finite dimensional vector
$K$-space and there exists $r_{0}\in\mathbb{Z}$ such that $M_{r} =
{\cal{D}}_{r-r_{0}}M_{r_{0}}$ (respectively, $M_{r} =
{\cal{E}}_{r-r_{0}}M_{r_{0}}$) for any $r \geq r_{0}$. Clearly, if
$M$ is generated as a ${\cal{D}}$- (respectively, ${\cal{E}}$-)
module by elements $x_{1},\dots x_{s}$, then
$\left(\sum_{i=1}^{s}{\cal{D}}_{r}x_{i}\right)_{r\in\mathbb{Z}}$
(respectively,
$\left(\sum_{i=1}^{s}{\cal{E}}_{r}x_{i}\right)_{r\in\mathbb{Z}}$) is
an excellent filtration of $M$; it is said to be the natural
filtration associated with the set of generators $\{x_{1},\dots,
x_{s}\}$.

If $M'$ and $M''$ are  $\Delta$-$\sigma$-$K$- (respectively,
$\Delta$-$\sigma^{\ast}$-$K$-) modules, then a mapping
$f:M'\rightarrow M''$ is said to be a $\Delta$-$\sigma$-homomorphism
if it is a homomorphism of ${\cal{D}}$- (respectively, ${\cal{E}}$-)
modules. If $M'$ and $M''$ are equipped with filtrations
$(M'_{r})_{r\in\mathbb{Z}}$ and $(M''_{r})_{r\in\mathbb{Z}}$,
respectively, and $f(M'_{r})\subseteq M''_{r}$ for every
$r\in\mathbb{Z}$, then $f$ is said to be a
$\Delta$-$\sigma$-homomorphism of filtered $\Delta$-$\sigma$-$K$-
(respectively, $\Delta$-$\sigma^{\ast}$-$K$-) modules.

The following two statements are direct consequences of
\cite[Theorem 6.7.3]{KLMP} and \cite[Theorem 6.7.10]{KLMP},
respectively.

\begin{theorem}
With the above notation, let $K$ be a $\Delta$-$\sigma$-field, $M$ a
finitely generated $\Delta$-$\sigma$-$K$-module, and
$(M_{r})_{r\in\mathbb{Z}}$ the natural filtration associated with
some finite system of generators of $M$ over the ring of
$\Delta$-$\sigma$-operators ${\cal{D}}$. Then there is a numerical
polynomial $\phi(t)\in\mathbb{Q}[t]$ such that:

\medskip

\noindent{\em(i)}\, $\phi(r) = \dim_{K}M_{r}$ for all sufficiently
large $r\in\mathbb{Z}$.

\smallskip

\noindent{\em(ii)}\, $\deg \phi \leq m+n$ and $\phi(t)$ can be
written as $\phi(t) = \D\sum_{i=0}^{m+n}a_{i}{t+i\choose i}$ where
$a_{0},\dots, a_{m+n}\in\mathbb{Z}$.

\smallskip

\noindent{\em(iii)}\, $d =  \deg \phi(t)$,\, $a_{n}$ and $a_{d}$ do
not depend on the finite set of generators of the ${\cal{D}}$-module
$M$ the filtration $(M_{r})_{r\in\mathbb{Z}}$ is associated with.
Furthermore, $a_{m+n}$ is equal to the $\Delta$-$\sigma$-dimension
of $M$ over $K$ (denoted by $\Delta$-$\sigma$-$\dim_{K}M$), that is,
to the maximal number of elements $x_{1},\dots,x_{k}\in M$ such that
the family $\{\lambda x_{i}\,|\,\lambda\in\Lambda, 1\leq i\leq k\}$
is linearly independent over $K$.
\end{theorem}

\begin{theorem}
Let $f: M'\rightarrow M''$ be an injective homomorphism of filtered
$\Delta$-$\sigma$-$K$-modules $M'$ and $M''$ with filtrations
$(M'_{r})_{r\in\mathbb{Z}}$ and $(M''_{r})_{r\in\mathbb{Z}}$,
respectively.  If the filtration of $M''$ is excellent, then the
filtration of $M'$ is excellent as well.
\end{theorem}

\smallskip

\centerline{PROOF OF THEOREM 2}

\medskip

Let $L = K\langle\eta_{1},\dots, \eta_{s}\rangle$ be a
$\Delta$-$\sigma$-field extension of a
$\Delta$-$\sigma^{\ast}$-field $K$. Let $L^{\ast}$ be the inversive
closure of $L$, that is, $L^{\ast} = K\langle\eta_{1},\dots,
\eta_{s}\rangle^{\ast}$. Let $M = \Omega_{L^{\ast}|K}$, the module
of K\"ahler differentials associated with the extension
$L^{\ast}/K$. Then $\Omega_{L^{\ast}|K}$ can be treated as a
$\Delta$-$\sigma^{\ast}$-$K$-module where the action of the elements
of $\Delta\bigcup\sigma^{\ast}$ is defined in such a way that
$\delta(d\zeta) = d\delta(\zeta)$ and $\alpha(d\zeta) =
d\alpha(\zeta)$ for any $\zeta\in L$, $\delta\in\Delta$, $\alpha\in
\sigma^{\ast}$ (see \cite{Johnson1} and \cite[Lemma 4.2.8]{Levin5}).

For every $r\in\mathbb{N}$, let $M_{r}$ denote the vector
$L^{\ast}$-subspace of $M$ generated by all elements $d\zeta$ where
$\zeta\in K(\D\bigcup_{i=1}^{s}\Lambda^{\ast}(r)\eta_{i})$. It is
easy to check that $(M_{r})_{r\in\mathbb{Z}}$ ($M_{r} = 0$ if $r <
0$) is the natural filtration of the
$\Delta$-$\sigma^{\ast}$-$L^{\ast}$-module $M$ associated with the
system of $\Delta$-$\sigma^{\ast}$-generators $\{d\eta_{1},\dots,
d\eta_{s}\}$ .

Let $F$ be any intermediate $\Delta$-$\sigma$-field of $L/K$, $F_{r}
= F\bigcap K(\{\lambda\eta_{j}\,|\,\lambda\in\Lambda(r), 1\leq j\leq
s\})$ ($r\in\mathbb{N}$) and $F_{r} = 0$ if $r<0$. Let ${\cal{E}}$
and ${\cal{D}}$ denote the ring of
$\Delta$-$\sigma^{\ast}$-operators over $L^{\ast}$ and the ring of
$\Delta$-$\sigma$-operators over $L$, respectively. Let $N$ be the
${\cal{D}}$-submodule of $M$ generated by all elements of the form
$d\zeta$ with $\zeta\in F$ (by $d\zeta$ we always mean
$d_{L^{\ast}|K}\zeta$). Furthermore, for any $r\in\mathbb{N}$, let
$N_{r}$ be the vector $L$-space generated by all elements $d\zeta$
with $\zeta\in F_{r}$ and $N_{r}=0$ if $r<0$.

It is easy to see that $(N_{r})_{r\in \mathbb{Z}}$ is a filtration
of the $\Delta$-$\sigma$-$L$-module $N$, and if $M' =
\sum_{i=1}^{s}{\cal{D}}d\eta_{i}$, then the embedding $N\rightarrow
M'$ is a homomorphism of filtered ${\cal{D}}$-modules. ($M'$ is
considered as a filtered ${\cal{D}}$-module with the excellent
filtration
$\left(\sum_{i=1}^{s}{\cal{D}}_{r}d\eta_{i}\right)_{r\in\mathbb{Z}}$.)
By Theorem 4, $(N_{r})_{r\in\mathbb{Z}}$ is an excellent filtration
of the ${\cal{D}}$-module $N$. Applying Theorem 3 we obtain that
there exists a polynomial $\chi_{K, F, \eta}(t)\in\mathbb{Q}[t]$
such that $\chi_{K, F, \eta}(t)(r) = \dim_{K}N_{r}$ for all
sufficiently large $r\in \mathbb{Z}$.

As it is shown in \cite[Chapter V, Section 23]{Morandi}, elements
$\zeta_{1},\dots, \zeta_{k}\in L^{\ast}$ are algebraically
independent over $K$ if and only if the elements $d\zeta_{1},\dots,
d\zeta_{k}$ are linearly independent over $L^{\ast}$. Thus, if
$\zeta_{1},\dots, \zeta_{k}\in F_{r}$ ($r\in \mathbb{Z}$) are
algebraically independent over $K$, then the elements
$d\zeta_{1},\dots, d\zeta_{k}\in N_{r}$ are linearly independent
over $L^{\ast}$ and therefore over $L$.  Conversely, if elements
$dx_{1},\dots, dx_{h}$ ($x_{i}\in F_{r}$ for $i=1,\dots, h$) are
linearly independent over $L$, then $x_{1},\dots, x_{h}$ are
algebraically independent over $K$. Otherwise, we would have a
polynomial $f(X_{1},\dots, X_{h})\in K[X_{1},\dots, X_{h}]$ of the
smallest possible degree such that $f(x_{1},\dots, x_{h})=0$. Then
$df(x_{1},\dots, x_{h}) = \sum_{i=1}^{h}\frac{\partial f}{\partial
X_{i}}(x_{1},\dots, x_{h})dx_{i} = 0$ where not all coefficients of
$dx_{i}$ are zeros (they are expressed by polynomials of degree less
than $\deg f$). Since all the coefficients lie in $L$, we would have
a contradiction with the linear independence of $dx_{1},\dots,
dx_{h}$ over $L$.

It follows that $\dim_{L}N_{r} = \trdeg_{K}F_{r}$ for all
$r\in\mathbb{N}$. Applying Theorem 3 we obtain the statement of
Theorem 2.\, $\square$

\medskip

Clearly, if $F=L$, then Theorem 2 implies Theorem 1. Note also that
if an intermediate field $F$ of a finitely generated
$\Delta$-$\sigma$-field extension $L/K$ is not a
$\Delta$-$\sigma$-subfield of $L$, there might be no numerical
polynomial whose values for sufficiently large integers $r$ are
equal to $\trdeg_{K}(F\bigcap
K(\{\lambda\eta_{j}\,|\,\lambda\in\Lambda(r), 1\leq j\leq s\}))$.
Indeed, let $\Delta = \{\delta\}$ and $\sigma = \emptyset$. Let $L =
K\langle y\rangle$, where the $\Delta$-$\sigma$-generator $y$ is
$\Delta$-$\sigma$-independent over $K$, and let $F = K(\delta^{2}y,
\dots, \delta^{2k}y,\dots )$. Then $\Lambda =
\{\delta^{i}\,|\,i\in\mathbb{N}\}$, $\Lambda(r) = \{1, \delta,\dots,
\delta^{r}\}$, $F_{r} = F\bigcap K(\lambda
y\,|\,\lambda\in\Lambda(r))$ and $\trdeg_{K}F_{r} = [{\frac{r}{2}}]$
(the integer part of ${\frac{r}{2}}$), which is not a polynomial of
$r$. In this case, the function $\phi(r) = \trdeg_{K}F_{r}$ is a
quasi-polynomial, but if one takes $F = K(\delta^{2}y, \dots,
\delta^{2^{k}}y,\dots )$, then $\trdeg_{K}F_{r} = [\log_{2}r]$.

\section{Type and dimension of difference-differential field extensions}

Let $K$ be an inversive difference-differential ($\Delta$-$\sigma$-)
field with a basic set $\Delta\bigcup\sigma$ where $\Delta =
\{\delta_{1},\dots, \delta_{m}\}$ and $\sigma = \{\alpha_{1},\dots,
\alpha_{n}\}$ are the sets of derivations and automorphisms of $K$,
respectively. Let $L = K\langle \eta_{1},\dots,\eta_{s}\rangle$ be a
$\Delta$-$\sigma$-field extension of $K$ generated by a finite set
$\eta =\{\eta_{1},\dots,\eta_{s}\}$. (We keep the notation
introduced in section 2.)

Let $\mathfrak{U}$ denote the set of all intermediate
$\Delta$-$\sigma$-fields of the extension $L/K$ and
$$\mathfrak{B}_{\mathfrak{U}} = \{(F, E)\in \mathfrak{U}\times
\mathfrak{U}\,|\,F\supseteq E\}.$$ Furthermore, let
$\overline{\mathbb{Z}}$ denote the ordered set $\mathbb{Z}\bigcup
\{\infty\}$ (where the natural order on $\mathbb{Z}$ is extended by
the condition $a < \infty$ for any $a\in\mathbb{Z}$).

\begin{proposition} With the above notation, there exists a unique
mapping\\
$\mu_{\mathfrak{U}}: \mathfrak{B}_{\mathfrak{U}}\rightarrow
\overline{\mathbb{Z}}$ such that

{\em (i)}\, $\mu_{\mathfrak{U}}(F, E)\geq -1$ for any pair $(F,
E)\in \mathfrak{B}_{\mathfrak{U}}$\,.

{\em (ii)}\, If $d\in\mathbb{N}$, then $\mu_{\mathfrak{U}}(F, E)\geq
d$ if and only if $\trdeg_{E}F > 0$ and there exists an infinite
descending chain of intermediate $\Delta$-$\sigma$-fields
\begin{equation} F = F_{0}\supseteq F_{1}\supseteq \dots
\supseteq F_{r}\supseteq \dots \supseteq E\end{equation} such that
\begin{equation} \mu_{\mathfrak{U}}(F_{i}, F_{i+1})\geq d-1\, \,\,\, (i = 0, 1,
\dots).\end{equation}
\end{proposition}
\begin{proof} In order to show the existence and uniqueness of the
desired mapping $\mu_{\mathfrak{U}}$, one can just mimic the proof
of the corresponding statement for chains of prime differential
ideals given in \cite[Section 1]{Johnson2} (see also
\cite[Proposition 4.1]{Levin4} and \cite[Section 4]{Levin6} where
similar arguments were applied to differential and inversive
difference field extensions, respectively). Namely, let us set
$\mu_{\mathfrak{U}}(F, E) = -1$ if $F=E$ or the field extension
$F/E$ is algebraic. If $(F, E)\in \mathfrak{B}_{\mathfrak{U}}$,
$\trdeg_{E}F > 0$ and for every $d\in\mathbb{N}$, there exists a
chain of intermediate $\Delta$-$\sigma$-fields (1) with condition
(2), we set $\mu_{\mathfrak{U}}(F, E) = \infty$. Otherwise, we
define $\mu_{\mathfrak{U}}(F, E)$ as the maximal integer $d$ for
which condition (ii) holds (that is, $\mu_{\mathfrak{U}}(F, E)\geq
d$). It is clear that the mapping $\mu_{\mathfrak{U}}$ defined in
this way is unique. \, $\square$
\end{proof}

With the notation of the last proposition, we define the {\em type}
of a $\Delta$-$\sigma$-field extension $L/K$ as the integer
\begin{equation} \type(L/K) =
\sup\{\mu_{\mathfrak{U}}(F, E)\,|\,(F, E)\in
\mathfrak{B}_{\mathfrak{U}}\}.
\end{equation}
and the {\em dimension} of the
$\Delta$-$\sigma$-extension $L/K$ as the number

\medskip

\noindent $\dim(L/K) = \sup\{q\in \mathbb{N}\,|\,\text{there exists
a chain}\,\, F_{0}\supseteq F_{1}\supseteq \dots \supseteq F_{q}$
such that $F_{i}\in \mathfrak{U}$ and
\begin{equation}
\mu_{\mathfrak{U}}(F_{i-1}, F_{i}) = \type(L/K) \hspace{0.2in} (i=
1,\dots, q)\}.
\end{equation}

\medskip

It is easy to see that for any pair of intermediate
$\Delta$-$\sigma$-fields of $L/K$ such that $(F, E)\in
\mathfrak{B}_{\mathfrak{U}}$, $\mu_{\mathfrak{U}}(F, E) = -1$ if and
only if the field extension $E/F$ is algebraic. It is also clear
that if $\type(L/K) < \infty$, then $\dim(L/K) > 0$.

\begin{proposition} With the above notation, let $F$ and $E$ be
intermediate $\Delta$-$\sigma$-fields of
a $\Delta$-$\sigma$-field extension $L =K\langle\eta_{1},\dots,
\eta_{s}\rangle$ generated by a finite set $\eta=\{\eta_{1},\dots,
\eta_{s}\}$. Let $F\supseteq E$, so that $(F, E)\in
\mathfrak{B}_{\mathfrak{U}}$. Then for any integer $d\geq -1$, the
inequality $\mu_{\mathfrak{U}}(F, E)\geq d$ implies the inequality
$\deg(\chi_{K, F,\eta}(t) - \chi_{K,E,\eta}(t))\geq d.$ ($\chi_{K,
F,\eta}(t)$ and $\chi_{K,E,\eta}(t)$ are the
$\Delta$-$\sigma$-dimensions polynomials of the fields $F$ and $E$
associated with the set of $\Delta$-$\sigma$-generators $\eta$ of
$L/K$.)
\end{proposition}

\begin{proof} We proceed by induction on $d$. Since
$\deg(\chi_{K,F,\eta}(t) - \chi_{K,E,\eta}(t))\geq -1$ for any pair
$(F, E)\in \mathfrak{B}_{\mathfrak{U}}$ and $\deg(\chi_{K,F,\eta}(t)
- \chi_{K,E,\eta}(t))\geq 0$ if $\trdeg_{E}F > 0$, our statement is
true for $d=-1$ and $d=0$. (As usual we assume that the degree of
the zero polynomial is $-1$.)

Let $d > 0$ and let  the statement be true for all nonnegative
integers less than $d$. Let $\mu_{\mathfrak{U}}(F, E)\geq d$ for
some pair $(F, E)\in \mathfrak{B}_{\mathfrak{U}}$, so that there
exists a chain of intermediate $\Delta$-$\sigma$-fields (1) such
that $\mu_{\mathfrak{U}}(F_{i}, F_{i+1})\geq d-1$ ($i=0, 1,\dots $).
If $\deg(\chi_{K,F_{i},\eta}(t) - \chi_{K,F_{i+1},\eta}(t))\geq d$
for some $i\in\mathbb{N}$, then $\deg(\chi_{K,F,\eta}(t) -
\chi_{K,E,\eta}(t))\geq \deg(\chi_{K,F_{i},\eta}(t) -
\chi_{K,F_{i+1},\eta}(t))\geq d$, so the statement of the
proposition is true.

Suppose that $\deg(\chi_{K,F_{i},\eta}(t) -
\chi_{K,F_{i+1},\eta}(t)) = d-1$ for every $i\in\mathbb{N}$, that
is, $\chi_{K,F_{i},\eta}(t) - \chi_{K,F_{i+1},\eta}(t) =
\D\sum_{j=0}^{d-1}a_{j}^{(i)}{{t+j}\choose j}$ where
$a_{0}^{(1)},\dots, a_{d-1}^{(i)}\in\mathbb{Z}$, $a_{d-1}^{(i)}> 0$.
Then $$\chi_{K,F,\eta}(t) - \chi_{K,F_{i+1},\eta}(t) =
\sum_{k=0}^{i}(\chi_{K,F_{k},\eta}(t) - \chi_{K,F_{k+1},\eta}(t)) =
\sum_{j=0}^{d-1}b_{j}^{(i)}{{t+j}\choose j}$$ where
$b_{0}^{(i)},\dots, b_{d-1}^{(i)}\in\mathbb{Z}$ and $b_{d-1}^{(i)} =
\sum_{k=0}^{i}a_{d-1}^{(k)}$. Therefore, $b_{d-1}^{(0)} <
b_{d-1}^{(1)} <\dots $ and $\lim_{i\rightarrow \infty}b_{d-1}^{(i)}
= \infty$. On the other hand, $\deg(\chi_{K,F,\eta}(t) -
\chi_{K,F_{i+1},\eta}(t))\leq \deg(\chi_{K,F,\eta}(t) -
\chi_{K,E,\eta}(t))$. If  $\deg(\chi_{K,F,\eta}(t) -
\chi_{K,E,\eta}(t)) = d-1$, that is, $\chi_{K,F,\eta}(t) -
\chi_{K,E,\eta}(t)= \D\sum_{j=0}^{d-1}c_{j}{{t+j}\choose j}$ for
some $c_{0},\dots, c_{d-1}\in\mathbb{Z}$, then we would have
$b_{d-1}^{(i)} < c_{d-1}$ for all $i\in\mathbb{N}$ contrary to the
fact that $\lim_{i\rightarrow \infty}b_{d-1}^{(i)} = \infty$. Thus,
$\deg(\chi_{K,F,\eta}(t) - \chi_{K,E,\eta}(t))\geq d$, so the
proposition is proved. \, $\square$
\end{proof}

The following theorem  provides a relationship between the
introduced characteristics of a finitely generated
$\Delta$-$\sigma$-extension and the invariants of its
$\Delta$-$\sigma$-dimension polynomial introduced by Theorem 2.

\begin{theorem}
Let $K$ be an inversive difference-differential ($\Delta$-$\sigma$-)
field with basic set $\Delta\bigcup\sigma$ where $\Delta =
\{\delta_{1},\dots, \delta_{m}\}$ and $\sigma = \{\alpha_{1},\dots,
\alpha_{n}\}$ are the sets of derivations and automorphisms of $K$,
respectively.  Let $L$ be a finitely generated
$\Delta$-$\sigma$-field extension of $K$. Then

\noindent{\em (i)}\, $\type(L/K)\leq\Delta$-$\sigma$-$\type_{K}L\leq
m+n$.

\noindent{\em (ii)}\, If $\Delta$-$\sigma$-$\trdeg_{K}L > 0$, then
$\type(L/K) = m+n$, $\dim(L/K) = \Delta$-$\sigma$-$\trdeg_{K}L$.

\noindent{\em (iii)}\, If $\Delta$-$\sigma$-$\trdeg_{K}L = 0$, then
$\type(L/K) < m+n$.
\end{theorem}

\begin{proof}
Let $\eta = \{\eta_{1},\dots, \eta_{s}\}$ be a system of
$\Delta$-$\sigma$-generators of $L$ over $K$ and for every
$r\in\mathbb{N}$, let $L_{r} =
K(\{\lambda\eta_{i}\,|\,\lambda\in\Lambda(r),  1\leq i\leq s\})$.
Furthermore, if $F$ is any intermediate $\Delta$-$\sigma$-field of
the extension $L/K$, then $F_{r}$ ($r\in\mathbb{N}$) will denote the
field $F\bigcap L_{r}$. By Theorem 2, there is a polynomial
$\chi_{K, F, \eta}(t)\in \mathbb{Q}[t]$ such that
$\chi_{K,F,\eta}(r) = \trdeg_{K}F_{r}$ for all sufficiently large
$r\in\mathbb{N}$, $\deg\,\chi_{K, F, \eta}\leq m+n$, and this
polynomial can be written as $\chi_{K, F, \eta}(t) =
\sum_{i=1}^{m+n}a_{i}{t+i\choose i}$ where $a_{0},\dots, a_{m+n}\in
\mathbb{Z}$ and $a_{m+n} = \Delta$-$\sigma$-$\trdeg_{K}F$. Clearly,
if $E$ and $F$ are two intermediate $\Delta$-$\sigma$-fields of
$L/K$ and $F\supseteq E$, then $\chi_{K,F,\eta}(t)\geq
\chi_{K,E,\eta}(t)$. (This inequality means that $\chi_{F}(r)\geq
\chi_{E}(r)$ for all sufficiently large $r\in \mathbb{N}$. As it is
first shown in \cite{Sit}, the set $W$ of all differential dimension
polynomials of finitely generated differential field extensions is
well ordered with respect to this ordering. At the same time, as it
is proved in \cite[Chapter 2]{KLMP}, $W$ is also the set of all
$\Delta$-$\sigma$-dimension polynomials associated with finitely
generated $\Delta$-$\sigma$-field extensions).

Note that if $F\supseteq E$ and $\chi_{K, F,\eta}(t) = \chi_{K, E,
\eta}(t)$, then the field extension $F/E$ is algebraic. Indeed, if
$x\in F$ is transcendental over $E$, then there exists
$r_{0}\in\mathbb{N}$ such that $x\in F_{r}$ for all $r\geq r_{0}$.
Therefore, $\trdeg_{K}F_{r} = \trdeg_{K}E_{r} + \trdeg_{E_{r}}F_{r}
> \trdeg_{K}E_{r}$ for all $r\geq r_{0}$ hence $\chi_{K, F,\eta}(t)>
\chi_{K, E,\eta}(t)$ contrary to our assumption.

Since $\deg(\chi_{K,F,\eta}(t) - \chi_{K,E,\eta}(t))\leq m+n$ for
any pair $(F, E)\in \mathfrak{B}_{\mathfrak{U}}$, the last
proposition implies that $\type(L/K)\leq\Delta$-$\sigma$-$\type_{K}L
\leq m+n$. If $\Delta$-$\sigma$-$\trdeg_{K}L = 0$, then
$\type(L/K)\leq\Delta$-$\sigma$-$\type_{K}L < m+n$. Thus, it remains
to prove statement (ii) of the theorem.

Let $\Delta$-$\sigma$-$\trdeg_{K}L > 0$, let element $x\in L$ be
$\Delta$-$\sigma$-transcendental over $K$ and let $F = K\langle
x\rangle$.  Clearly, in order to prove that $\type(L/K) = m+n$ it is
sufficient to show that $\mu_{\mathfrak{U}}(F, K)\geq m+n$. This
inequality, in turn, immediately follows from the consideration of
the following $m+n$ strictly descending chains of intermediate
$\Delta$-$\sigma$-fields of $F/K$.

\medskip

\noindent$F=K\langle x\rangle\supset K\langle\delta_{1}
x\rangle\supset K\langle\delta_{1}^{2}x\rangle\supset\dots\supset
K\langle\delta_{1}^{i_{1}}x\rangle\supset
K\langle\delta_{1}^{i_{1}+1}x\rangle\supset\dots \supset K$\,,

\medskip

\noindent$K\langle\delta_{1}^{i_{1}}x\rangle\supset
K\langle\delta_{1}^{i_{1}+1}x,
\delta_{1}^{i_{1}}\delta_{2}x\rangle\supset
K\langle\delta_{1}^{i_{1}+1}x,
\delta_{1}^{i_{1}}\delta_{2}^{2}x\rangle\supset\dots
K\langle\delta_{1}^{i_{1}+1}x,
\delta_{1}^{i_{1}}\delta_{2}^{i_{2}}x\rangle\supset
K\langle\delta_{1}^{i_{1}+1}x,
\delta_{1}^{i_{1}}\delta_{2}^{i_{2}+1}x\rangle\supset\dots\supset
K\langle\delta_{1}^{i_{1}+1}x\rangle$\,,

\medskip

\centerline{$\dots$}

\medskip

\noindent $K\langle\delta_{1}^{i_{1}+1}x,
\delta_{1}^{i_{1}+1}\delta_{2}^{i_{2}+1}x,\dots,
\delta_{1}^{i_{1}+1}\dots\delta_{m-1}^{i_{m-1}+1}x,
\delta_{1}^{i_{1}+1}\dots\delta_{m-1}^{i_{m-1}+1}\delta_{m}^{i_{m}}x\rangle\supset
K\langle\delta_{1}^{i_{1}+1}x, $


\noindent
$,\dots,\delta_{1}^{i_{1}+1}\dots\delta_{m-1}^{i_{m-1}+1}x,
\delta_{1}^{i_{1}+1}\dots\delta_{m-1}^{i_{m-1}+1}\delta_{m}^{i_{m}+1}x,
\delta_{1}^{i_{1}+1}\dots\delta_{m-1}^{i_{m-1}+1}\delta_{m}^{i_{m}}(\alpha_{1}-1)x\rangle\supset$


\noindent $\supset\dots\supset K\langle\delta_{1}^{i_{1}+1}x, \dots,
\delta_{1}^{i_{1}+1}\dots\delta_{m-1}^{i_{m-1}+1}x,
\delta_{1}^{i_{1}+1}\dots\delta_{m-1}^{i_{m-1}+1}\delta_{m}^{i_{m}}(\alpha_{1}-1)^{2}x\rangle\supset$


\noindent $\supset\dots\supset K\langle\delta_{1}^{i_{1}+1}x, \dots,
\delta_{1}^{i_{1}+1}\dots\delta_{m-1}^{i_{m-1}+1}x,
\delta_{1}^{i_{1}+1}\dots\delta_{m-1}^{i_{m-1}+1}\delta_{m}^{i_{m}}(\alpha_{1}-1)^{i_{m+1}}x\rangle\supset$


\noindent $\dots\supset K\langle\delta_{1}^{i_{1}+1}x, \dots,
\delta_{1}^{i_{1}+1}\dots\delta_{m-1}^{i_{m-1}+1}x,
\delta_{1}^{i_{1}+1}\dots\delta_{m-1}^{i_{m-1}+1}\delta_{m}^{i_{m}}(\alpha_{1}-1)^{i_{m+1}+1}x\rangle\supset$

\noindent $\dots\supset K\langle\delta_{1}^{i_{1}+1}x,
\delta_{1}^{i_{1}+1}\delta_{2}^{i_{2}+1}x,\dots,
\delta_{1}^{i_{1}+1}\dots\delta_{m-1}^{i_{m-1}+1}x,
\delta_{1}^{i_{1}+1}\dots\delta_{m-1}^{i_{m-1}+1}\delta_{m}^{i_{m}+1}x\rangle$\,,

\medskip

\centerline{$\dots$}

\medskip

\noindent $K\langle\delta_{1}^{i_{1}+1}x,
\dots,\delta_{1}^{i_{1}+1}\dots\delta_{m-1}^{i_{m-1}+1}
\delta_{m}^{i_{m}+1}(\alpha_{1}-1)^{i_{m+1}+1}\dots
(\alpha_{n-1}-1)^{i_{m+n-1}}x\rangle\supset$


\noindent $K\langle\delta_{1}^{i_{1}+1}x, \dots,
\delta_{1}^{i_{1}+1}\dots\delta_{m-1}^{i_{m-1}+1}x,
\delta_{1}^{i_{1}+1}\dots\delta_{m-1}^{i_{m-1}+1}\delta_{m}^{i_{m}}(\alpha_{1}-1)^{i_{m+1}+1}\dots\\
(\alpha_{n-1}-1)^{i_{m+n-1}+1}(\alpha_{n}-1)x\rangle\supset\dots\supset
K\langle\delta_{1}^{i_{1}+1}x, \dots,
\delta_{1}^{i_{1}+1}\dots\delta_{m-1}^{i_{m-1}+1}x,
\delta_{1}^{i_{1}+1}\dots\\ \delta_{m-1}^{i_{m-1}+1}
\delta_{m}^{i_{m}}(\alpha_{1}-1)^{i_{m+1}+1}\dots
(\alpha_{n-1}-1)^{i_{m+n-1}+1}(\alpha_{n}-1)^{i_{m+n}}x\rangle\supset\dots\supset
K\langle\delta_{1}^{i_{1}+1}x,$

\noindent$ \dots,\delta_{1}^{i_{1}+1}\dots\delta_{m-1}^{i_{m-1}+1}x,
\delta_{1}^{i_{1}+1}\dots\delta_{m-1}^{i_{m-1}+1}\delta_{m}^{i_{m}}(\alpha_{1}-1)^{i_{m+1}+1}\dots
(\alpha_{n-1}-1)^{i_{m+n-1}+1}x\rangle$\,.

\medskip

These $m+n$ chains show that $\mu_{\mathfrak{U}}(F, K)\geq m+n$,
hence $\type(L/K) = m+n$. Furthermore, if
$\Delta$-$\sigma$-$\trdeg_{K}L = k > 0$ and $x_{1},\dots, x_{k}$ is
a $\Delta$-$\sigma$-transcendence basic of $L$ over $K$, then every
$x_{i}$ ($2\leq i\leq k$) is $\Delta$-$\sigma$-independent over
$K\langle x_{1},\dots, x_{i-1}\rangle$. Therefore, the above chains
show that $\mu_{\mathfrak{U}}(K\langle x_{1}\rangle, K) =
\mu_{\mathfrak{U}}(K\langle x_{1}, x_{2}\rangle, K\langle
x_{1}\rangle)=\dots = \mu_{\mathfrak{U}}(K\langle x_{1},\dots,
x_{k}\rangle, K\langle x_{1},\dots, x_{k-1}\rangle) = m+n,$ hence
$\dim(L/K)\geq k = \Delta$-$\sigma$-$\trdeg_{K}L$.

In order to prove the opposite inequality, suppose that
$F_{0}\supseteq F_{1}\supseteq \dots \supseteq F_{p}$ is an
ascending chain of intermediate $\Delta$-$\sigma$-fields of the
extension $L/K$ such that $\mu_{\mathfrak{U}}(F_{i}, F_{i+1}) =
\type(L/K) = m+n$ \,for\, $i= 0,\dots, p-1$. Clearly, in order to
prove our inequality, it is sufficient to show that $p\leq k$.

For every $i=0,\dots, p$, the $\Delta$-$\sigma$-dimension polynomial
$\chi_{K, F_{i}, \eta}(t)$, whose existence is established by
Theorem 2, can be written as $\chi_{K, F_{i},\eta}(t) =
\D\sum_{j=0}^{m+n}a_{j}^{(i)}{{t+j}\choose j}$ where $a_{j}^{(i)}\in
\mathbb{Z}$ ($0\leq i\leq p-1$, $0\leq j\leq m+n$). Then $\chi_{K,
F_{0},\eta}(t) - \chi_{K, F_{p},\eta}(t) = \D\sum_{i=1}^{p}(\chi_{K,
F_{i-1},\eta}(t) - \chi_{K, F_{i},\eta}(t)) =
\D\sum_{i=1}^{p}\sum_{j=0}^{m+n}(a_{j}^{(i-1)} - a_{j}^{(i)})
{{t+j}\choose j} = \\$ $(a_{m+n}^{(0)} - a_{m+n}^{(p)})
\D{{t+m+n}\choose m+n} + o(t^{m+n})$ where $o(t^{m+n})$ denotes a
polynomial of degree at most $m+n-1$.

Since $\mu_{\mathfrak{U}}(F_{i}, F_{i+1}) = m+n$ ($0\leq i\leq
p-1$), we have $\deg(\chi_{K, F_{i},\eta}(t) - \chi_{K,
F_{i+1},\eta}(t)) = m+n$ (see Proposition 3). Therefore,
$a_{m+n}^{(0)}
> a_{m+n}^{(1)} > \dots > a_{m+n}^{(p)}$, hence
$$a_{m+n}^{(0)} - a_{m+n}^{(q)} = \sum_{i=1}^{p}(a_{m+n}^{(i-1)} -
a_{m+n}^{(i)})\geq p.$$ On the other hand, $\chi_{K, F_{0},\eta}(t)
- \chi_{K, F_{p},\eta}(t)\leq \chi_{K, L, \eta}(t) =
\D\sum_{i=0}^{m+n}a_{i}{t+i\choose i}$ where $a_{m+n} =
\Delta$-$\sigma$-$\trdeg_{K}L$.  Therefore, $p\leq a_{m+n}^{(0)} -
a_{m+n}^{(p)}\leq k = \sigma$-$\trdeg_{K}L$. This completes the
proof of the theorem. \, $\square$
\end{proof}

\section{Multivariate dimension polynomials of intermediate
$\sigma^{\ast}$-field extensions}

In this section we present a result that generalizes both Theorem 2
and the theorem on multivariate dimension polynomial of a finitely
generated differential field extension associated with a partition
of the basic set of derivations, see \cite[Theorem 4.6]{Levin2a}.

Let $K$ be a difference-differential ($\Delta$-$\sigma$-) field with
basic sets $\Delta = \{\delta_{1},\dots, \delta_{m}\}$ and $\sigma =
\{\alpha_{1},\dots, \alpha_{n}\}$ of derivations and automorphisms,
respectively. Suppose that these sets are represented as the unions
of $p$ and $q$ nonempty disjoint subsets, respectively ($p, q\geq
1$):
\begin{equation}\label{eq:8}\Delta = \Delta_{1}\bigcup \dots \bigcup \Delta_{p}, \,\,\, \sigma = \sigma_{1}\bigcup \dots \bigcup \sigma_{q}, \end{equation}
$\Delta_{1} = \{\delta_{1},\dots, \delta_{m_{1}}\},\, \Delta_{2} =
\{\delta_{m_{1}+1},\dots, \delta_{m_{1}+m_{2}}\}, \,\dots,
\Delta_{p} = \{\delta_{m_{1}+\dots + m_{p-1}+1},$\\$\dots,
\delta_{m}\}$, \,\, $\sigma_{1} = \{\alpha_{1},\dots,
\alpha_{n_{1}}\},\, \sigma_{2} = \{\alpha_{n_{1}+1},\dots,
\alpha_{n_{1}+n_{2}}\}, \,\dots, $\\$\sigma_{q} =
\{\alpha_{n_{1}+\dots + n_{q-1}+1},\dots, \alpha_{n}\};$
($m_{1}+\dots + m_{p} = m;\, n_{1} + \dots + n_{q} = n$).

If \,\,$\lambda  = \delta_{1}^{k_{1}}\dots
\delta_{m}^{k_{m}}\alpha_{1}^{l_{1}}\dots
\alpha_{n}^{l_{n}}\in\Lambda$ ($k_{i}, l_{j}\in\mathbb{N}$; we use
the notation of section 2), then the order of $\lambda$ with respect
to a set $\Delta_{i}$ ($1\leq i\leq p$) is defined as $\sum_{\mu =
m_{1}+\dots + m_{i-1}+1}^{m_{1}+\dots + m_{i}}k_{\mu}$; it is
denoted by $\ord_{i}\lambda$. (If $i=1$, the last sum is replaced by
$\sum_{\mu=1}^{m_{1}}k_{\mu}$.) Similarly, the order of $\lambda$
with respect to a set $\sigma_{j}$ ($1\leq j\leq q$) is defined as
$\sum_{\nu = n_{1}+\dots + n_{j-1}+1}^{n_{1}+\dots + n_{j}}l_{\nu}$;
it is denoted by $\ord'_{j}\lambda$. (If $j=1$, the last sum is
replaced by $\sum_{\nu=1}^{n_{1}}l_{\nu}$.) If $r_{1},\dots,
r_{p+q}\in\mathbb{N}$, we set \,\,\, $\Lambda(r_{1},\dots, r_{p+q})
=\\$ \noindent$ \{\lambda\in\Lambda\,|\,\ord_{i}\lambda\leq r_{i}$
($1\leq i\leq p$) and $\ord'_{j}\lambda\leq r_{p+j}$ $(1\leq j\leq
q)\}.$

\smallskip

Furthermore, for any permutation $(j_{1},\dots, j_{p+q})$ of the set
$\{1,\dots, p+q\}$, let $<_{j_{1},\dots, j_{p+q}}$ be the
lexicographic order on $\mathbb{N}^{p+q}$ such that $(r_{1},\dots,
r_{p+q})<_{j_{1},\dots, j_{p+q}} (s_{1},\dots, s_{p+q})$ if and
only if either $r_{j_{1}} < s_{j_{1}}$ or there exists
$k\in\mathbb{N}$, $1\leq k\leq p+q$, such that $r_{j_{\nu}} =
s_{j_{\nu}}$ for $\nu = 1,\dots, k$ and $r_{j_{k+1}} < s_{j_{k+1}}$.

If $A\subseteq\mathbb{N}^{p+q}$, then $A'$ will denote the set of
all $(p+q)$-tuples $a\in A$ that are maximal elements of this set
with respect to one of the $(p+q)!$ orders $<_{j_{1},\dots,
j_{p+q}}$. Say, if $A = \{(1, 1, 1), (2, 3, 0), (0, 2, 3), (2, 0,
5),$ $(3, 3, 1), (4, 1, 1), (2, 3, 3)\} \subseteq \mathbb{N}^{3}$,
then $A' = \{(2, 0, 5), (3, 3, 1), (4, 1, 1), (2, 3, 3)\}$.

\begin{theorem} With the above notation, let $F$ be an
intermediate $\Delta$-$\sigma$-field of a $\Delta$-$\sigma$-field
extension $L=K\langle\eta_{1},\dots,\eta_{s}\rangle$ generated by a
finite family $\eta = \{\eta_{1},\dots,\eta_{s}\}$. Let partitions
(5) be fixed and for any $r_{1},\dots, r_{p+q}\in\mathbb{N}^{p+q}$,
let
$$F_{r_{1},\dots, r_{p+q}} = F\bigcap K(\D\bigcup_{j=1}^{s}
\Lambda(r_{1},\dots, r_{p+q})\eta_{j}).$$ Then there exists a
polynomial in $p+q$ variables $\Phi_{K,
F,\eta}\in\mathbb{Q}[t_{1},\dots, t_{p+q}]$ such that

\noindent{\em (i)} \,$\Phi_{K, F,\eta}(r_{1},\dots, r_{p+q}) =
\trdeg_{K}K(\D\bigcup_{j=1}^{s} \Lambda(r_{1},\dots,
r_{p+q})\eta_{j})$

\noindent for all sufficiently large
$(r_{1},\dots,r_{p+q})\in\mathbb{N}^{p+q}$. (That is, there exist
$r_{1}^{(0)},\dots, r_{p+q}^{(0)}\in\mathbb{N}$ such that the
equality holds for all $(r_{1},\dots, r_{p+q})\in\mathbb{N}^{p+q}$
with $r_{i}\geq r_{i}^{(0)}$, $1\leq i\leq p+q$.)         ;

\noindent{\em (ii)} \, $\deg_{t_{i}}\Phi_{\eta} \leq m_{i}$ ($1\leq
i\leq p$),\, $\deg_{t_{p+j}}\Phi_{\eta} \leq n_{j}$ ($1\leq j\leq
q$) and $\Phi_{\eta}(t_{1},\dots, t_{p+q})$ can be represented as
\begin{equation}
\Phi_{\eta} = \D\sum_{i_{1}=0}^{m_{1}}\dots
\D\sum_{i_{p}=0}^{m_{p}}\D\sum_{i_{p+1}=0}^{n_{1}}\dots\D\sum_{i_{p+q}=0}^{n_{q}}a_{i_{1}\dots
i_{p+q}} {t_{1}+i_{1}\choose i_{1}}\dots {t_{p+q}+i_{p+q}\choose
i_{p+q}}\end{equation} where $a_{i_{1}\dots i_{p+q}}\in\mathbb{Z}$.

\noindent{\em (iii)} \,Let $E_{\eta} = \{(i_{1},\dots,
i_{p+q})\in\mathbb{N}^{p+q}\,|\, 0\leq i_{k}\leq m_{k}$ for
$k=1,\dots, p$, $0\leq i_{p+j}\leq n_{j}$ for $j=1,\dots, q$, and
$a_{i_{1}\dots i_{p+q}}\neq 0\}$.  Then $d = \deg\,\Phi_{\eta}$,
$a_{m_{1}\dots m_{p}n_{1}\dots n_{q}}$, elements $(k_{1},\dots,
k_{p+q})\in E_{\eta}'$, the corresponding coefficients
$a_{k_{1}\dots k_{p+q}}$, and the coefficients of the terms of total
degree $d$ do not depend on the choice of the set of
$\Delta$-$\sigma$-generators $\eta$. Furthermore, $a_{m_{1}\dots
m_{p}n_{1}\dots n_{q}} = \Delta$-$\sigma$-$\trdeg_{K}L$.
\end{theorem}

\begin{proof} We will mimic the method of the proof of Theorem 2 using
the results on multivariate dimension polynomials of
$\Delta$-$\sigma$-$L$-modules. Let $\cal{D}$ be the ring of
$\Delta$-$\sigma$-operators over $L$ considered as a filtered ring
with $(p+q)$-dimensional filtration $\{{\cal{D}}_{r_{1}, \dots,
r_{p+q}}\,|\,(r_{1}, \dots, r_{p+q})\in \mathbb{Z}^{p+q}\}$ where
for any $r_{1}, \dots, r_{p+q}\in \mathbb{N}^{p+q}$,
${\cal{D}}_{r_{1}, \dots, r_{p+q}}$ is the vector $L$-subspace of
$\cal{D}$ generated by $\Lambda(r_{1}, \dots, r_{p+q})$, and
${\cal{D}}_{r_{1}, \dots, r_{p+q}} = 0$ if at least one $r_{i}$ is
negative. If $M$ is a $\Delta$-$\sigma$-$L$-module, then a family
$\{M_{r_{1}, \dots, r_{p+q}} | (r_{1}, \dots,
r_{p+q})\in\mathbb{Z}^{p+q}\}$ of vector $K$-subspaces of $M$ is
said to be a $(p+q)$-dimensional filtration of $M$ if

\noindent(i)\, $M_{r_{1}, \dots, r_{p+q}}\subseteq M_{s_{1}, \dots,
s_{p+q}}$ whenever $r_{i}\leq s_{i}$ for $i=1,\dots, p+q$.

\noindent(ii)\, $\bigcup_{(r_{1}, \dots,
r_{p+q})\in\mathbb{Z}^{p+q}}M_{r_{1}, \dots, r_{p+q}} = M$.

\noindent(iii)\, There exists $(r^{(0)}_{1},\dots,
r^{(0)}_{p+q})\in\mathbb{Z}^{p}$ such that $M_{r_{1}, \dots,
r_{p+q}} = 0$ if $r_{i} < r^{(0)}_{i}$ for at least one index $i$.

\noindent(iv)\, ${\cal{D}}_{r_{1}, \dots, r_{p+q}}M_{s_{1}, \dots,
s_{p+q}}\subseteq M_{r_{1}+s_{1}, \dots, r_{p+q}+s_{p+q}}$ for any
$(p+q)$-tuples $(r_{1}, \dots, r_{p+q})$, $(s_{1}, \dots,
s_{p+q})\in\mathbb{Z}^{p+q}$,

\smallskip

If every vector $L$-space $M_{r_{1}, \dots, r_{p+q}}$ is
finite-dimensional and there exists an element $(h_{1}, \dots,
h_{p})\in\mathbb{Z}^{p}$ such that ${\cal{D}}_{r_{1}, \dots,
r_{p+q}}M_{h_{1}, \dots, h_{p+q}} = M_{r_{1}+h_{1}, \dots,
r_{p+q}+h_{p+q}}$ for any $(r_{1}, \dots,
r_{p+q})\in\mathbb{N}^{p+q}$, the filtration $\{M_{r_{1}, \dots,
r_{p+q}} | (r_{1}, \dots, r_{p+q})\in \mathbb{Z}^{p+q}\}$ is called
{\em excellent}. Clearly, if $z_{1},\dots, z_{k}$ is a finite system
of generators of a $\Delta$-$\sigma$-$L$-module $M$, then
$\{\sum_{i=1}^{k}{\cal{D}}_{r_{1}, \dots, r_{p+q}}z_{i} | (r_{1},
\dots, r_{p+q})\in \mathbb{Z}^{p+q}\}$  is an excellent
$(p+q)$-dimensional filtration of $M$.

\medskip

As we have seen, the module of K\"ahler differentials $\Omega_{L|K}$
can be treated as a $\Delta$-$\sigma$-$L$-module such that
$\beta(d\zeta) = d\beta(\zeta)$ for any $\zeta\in L$,
$\beta\in\Delta\bigcup\sigma$.

Furthermore, $\Omega_{L|K} = \sum_{i=1}^{s}\mathcal{D}d\eta_{i}$,
and if $(\Omega_{L|K})_{r_{1}\dots r_{p+q}}$\, ($r_{1},\dots ,
r_{p+q}\in \mathbb{N}$) is the vector $L$-subspace of $\Omega_{L|K}$
generated by the set $\{d\eta| \eta \in
K(\{\lambda\eta_{j}\,|\,\lambda\in\Lambda(r_{1},\dots, r_{p+q}),
1\leq j\leq s\})\}$ and $(\Omega_{L|K})_{r_{1}\dots r_{p+q}}=0$
whenever at least one $r_{i}$ is negative, then
$\{(\Omega_{L|K})_{r_{1}\dots r_{p+q}}\,|\,(r_{1},\dots ,
r_{p+q})\in\mathbb{Z}^{p+q}\}$ is an excellent $(p+q)$-dimensional
filtration of $\Omega_{L|K}$.

Let $N$ be the ${\cal{E}}$-submodule of $\Omega_{L|K}$ generated by
all elements $d\zeta$ where $\zeta\in F$. For any $r_{1},\dots,
r_{p+q}\in\mathbb{N}$, let $N_{r_{1},\dots, r_{p+q}}$ be the vector
$L$-space generated by all elements $d\zeta$ where $\zeta\in
F_{r_{1},\dots, r_{p+q}}$. Setting $N_{r_{1},\dots, r_{p+q}} = 0$ if
$(r_{1},\dots,
r_{p+q})\in\mathbb{Z}^{p+q}\setminus\mathbb{N}^{p+q}$, we get a
$(p+q)$-dimensional filtration of the $\Delta$-$\sigma$-$L$-module
$N$, and the embedding $N\rightarrow \Omega_{L|K}$ becomes a
homomorphism of $(p+q)$-filtered $\Delta$-$\sigma$-$L$-modules. Now,
one can mimic the proof of Theorem 3.2.8 of \cite{Levin5} to show
that the filtration $\{N_{r_{1}, \dots, r_{p+q}} | (r_{1}, \dots,
r_{p+q})\in \mathbb{Z}^{p+q}\}$ is excellent. The result of Theorem
6 immediately follows from the fact that $\dim_{L}N_{r_{1}, \dots,
r_{p+q}} = \trdeg_{K}F_{r_{1},\dots, r_{p+q}}$ for all
$(r_{1},\dots, r_{p+q})\in\mathbb{N}^{p+q}$ (as it is mentioned in
the proof of Theorem 2, a family $(\zeta_{i})_{i\in I}$ of elements
of $L$ (in particular, of $F_{r_{1}, \dots, r_{p+q}}$)  is
algebraically independent over $K$ if and only if the family
$(d\zeta_{i})_{i\in I}$ is linearly independent over $L$) and the
result of \cite[Theorem 3.5.8]{Levin5} (it states that under the
above conditions, there exists a polynomial $\Phi_{K, F,
\eta}(t_{1},\dots, t_{p+q})\in \mathbb{Q}[t_{1},\dots, t_{p+q}]$
such that $\Phi_{\eta}(r_{1},\dots, r_{p+q}) = \dim_{L}N_{r_{1},
\dots, r_{p+q}}$ for all sufficiently large $(r_{1}, \dots,
r_{p+q})\in \mathbb{Z}^{p+q}$ and $\Phi_{K, F, \eta}(t_{1},\dots,
t_{p+q})$ satisfies conditions (ii) of Theorem 6. Statement (iii) of
Theorem 6 can be obtained in the same way as statement (iii) of
Theorem 2 of \cite{Levin6}.)
\end{proof}


\begin{thebibliography}{}

\bibitem[1]{E}
Einstein, A.: The Meaning of Relativity. Appendix II (Generalization
of gravitation theory), 4th edn. Princeton, 1953, 133--165.

\bibitem[2]{Johnson1}
Johnson, Joseph L.: K\"ahler differentials and
differential algebra.  Ann. of Math. (2), 89 (1969), 92--98.

\bibitem[3]{Johnson2}
Johnson, Joseph L.:  A notion on Krull dimension for differential
rings.  Comment. Math. Helv., 44 (1969), 207--216.

\bibitem[4]{K1}
Kolchin, E. R.: The notion of dimension in the theory of algebraic
differential equations.  Bull. Amer. Math. Soc., 70
(1964), 570--573.

\bibitem[5]{K2}
Kolchin, E. R.:  Differential Algebra and Algebraic Groups.
Acad. Press, New York, 1973.

\bibitem[6]{KLMP}
Kondrateva, M. V., Levin, A. B., Mikhalev, A. V., Pankratev, E. V.:
Differential and Difference Dimension Polynomials.
Dordrecht: Kluwer Academic Publishers, 1999.

\bibitem[7]{Levin1} Levin, A. B.: Characteristic polynomials of
filtered difference modules and difference field extensions.
Russian Math. Surveys, 33 (1978), no.3, 165--166.

\bibitem[8]{Levin2}
Levin, A. B.: Characteristic Polynomials of Inversive Difference
Modules and Some Properties of Inversive Difference Dimension.
Russian Math. Surveys, 35 (1980), no. 1, 217--218.

\bibitem[9]{Levin2a}
Levin, A. B.: Gr\"obner bases with respect to several orderings and
multivariable dimension polynomials.  J. Symbolic Comput., 42
(2007), no. 5, 561--578.

\bibitem[10]{Levin3}
Levin, A. B.: Difference Algebra.  Springer, New York, 2008.

\bibitem[11]{Levin4}
Levin, A. B.:  Dimension polynomials of intermediate fields and
Krull-type dimension of finitely generated differential field
extensions. Mathematics in Computer Science, 4 (2010), no.
2-3, 143--150.

\bibitem[12]{Levin5}
Levin, A. B.:  Multivariate dimension polynomials of inversive
difference field extensions.  Algebraic and algorithmic aspects
of differential and integral operators, 146--163. Lecture
Notes in Comput. Sci., 8372. Springer, Heidelberg, 2014.

\bibitem[13]{Levin6}
Levin A. B.:  Dimension Polynomials of Intermediate Fields of Inversive
Difference Field Extensions.  Mathematical Aspects of Computer
and Information Sciences. MACIS 2015. Lecture Notes in Comput.
Sci., vol 9582 (2016), 362--376

\bibitem[14]{Levin7}
Levin, A. B.: Multivariate Difference-Differential Polynomials and
New Invariants of Difference-Differential Field Extensions.
Proceedings of ISSAC 2013. Boston, MA, 267--274.

\bibitem[15]{LM}
Levin, A. B., Mikhalev A. V.: Difference-Differential Dimension
Polynomials. Moscow State University,  VINITI, no. 6848-B88
(1988), 1--64.

\bibitem[16]{MP}
Mikhalev, A. V.; Pankratev, E. V.: Differential dimension polynomial
of a system of differential equations. Algebra (collection of
papers). Moscow State Univ., Moscow, 1980, 57--67.

\bibitem[17]{Morandi}
Morandi, P.: Fields and galois Theory. Springer, New York, 1996.

\bibitem[18]{Sit}
Sit, W. Well-ordering of certain numerical polynomials. Trans.
Amer. Math. Soc., 212 (1975), 37--45.

\bibitem[19]{ZW1}
Zhou, M.; Winkler, F.: Computing difference-differential dimension
polynomials by relative Gr\"obner bases in difference-differential
modules.  J. Symbolic Comput., 43 (2008), no. 10, 726--745.

\bibitem[20]{ZW2}
Zhou, M.; Winkler, F.: Gr\"obner bases in difference-differential
modules and difference-differential dimension polynomials.  Sci.
China, Ser. A, 51 (2008), no. 9, 1732--1752.


\end{thebibliography}
\end{document}